\documentclass{amsart}
\usepackage{latexsym}
\usepackage{amstext}
\usepackage{amsmath}
\usepackage{amssymb}
\usepackage{amsthm}
\usepackage{url}

\newtheorem{theorem}{Theorem}
\newtheorem{lemma}[theorem]{Lemma}
\newtheorem{corollary}[theorem]{Corollary}

\theoremstyle{definition}

\def \Qbar {\overline{\mathbb{Q}}}
\def \O {\mathcal{O}}
\def \p {\mathfrak{p}}

\def \Q {\mathbb{Q}}

\DeclareMathOperator{\Jac}{Jac}
\DeclareMathOperator{\lcm}{lcm}

\DeclareMathOperator{\Cl}{Cl}

\DeclareMathOperator{\ord}{ord}

\DeclareMathOperator{\dv}{div}
\DeclareMathOperator{\rk}{rk}
\DeclareMathOperator{\im}{im}
\DeclareMathOperator{\tors}{tors}

\DeclareMathOperator{\divisor}{div}
\DeclareMathOperator{\Disc}{Disc}

\begin{document}
\bibliographystyle{amsplain}
\author[Aaron Levin]{Aaron Levin }
\thanks{The first author was supported in part by NSF grant DMS-1352407.}
\address{Department of Mathematics\\Michigan State University\\East Lansing, MI 48824, USA}
\email{adlevin@math.msu.edu}

\author[Yan Shengkuan]{Yan Shengkuan}
\address{School of Mathematics and Statistics\\Xi'an Jiaotong University\\Xi'an, 710049, China}
\email{sjwijd@stu.xjtu.edu.cn}

\author[Luke Wiljanen]{Luke Wiljanen}
\address{Department of Mathematics\\Michigan State University\\East Lansing, MI 48824, USA}
\email{wiljane8@msu.edu}

\title{Quadratic fields with a class group of large $3$-rank}
\date{}
\begin{abstract}
We prove that there are $\gg X^\frac{1}{30}/\log X$ imaginary quadratic number fields with an ideal class group of $3$-rank at least $5$ and discriminant bounded in absolute value by $X$.  This improves on an earlier result of Craig, who proved the infinitude of imaginary quadratic fields with an ideal class group of $3$-rank at least $4$.  The proofs rely on constructions of Mestre for $j$-invariant $0$ elliptic curves of large Mordell-Weil rank, and a method of the first author and Gillibert for constructing torsion in ideal class groups of number fields  from rational torsion in Jacobians of curves. We also consider analogous questions concerning rational $3$-torsion in hyperelliptic Jacobians.
\end{abstract}

\thanks{2010\ {\it Mathematics Subject Classification}: Primary 11R29; Secondary 11G30, 14H40}

\maketitle

\section{Introduction}

We study the problem of constructing and counting quadratic number fields with large $3$-rank (i.e., with an ideal class group of large $3$-rank). Our main result is the following:
\renewcommand{\arraystretch}{1.2}
\begin{theorem}
\label{mtheorem}
There exist $\gg X^\frac{1}{30}/\log X$ imaginary (resp.\ real) quadratic number fields $k$ with ${|\Disc(k)|<X}$ and
\begin{align*}
& \rk_3 \Cl(k)\geq 5 \\
\text{(resp. } &\rk_3 \Cl(k)\geq 4\text{).}
\end{align*}
\end{theorem}

In particular, we produce infinitely many imaginary quadratic fields with a class group of $3$-rank at least $r=5$, improving on the best previous result, due to Craig \cite{Craig1}, who in 1977 proved the same statement but with $r=4$.  In the real quadratic case, we obtain a quantitative version of a result of Diaz y Diaz \cite{Diaz}, who combined Craig's work with Scholz's reflection principle \cite{Scholz} to prove an analogous result giving the infinitude of real quadratic fields with $3$-rank at least $4$.  For smaller $3$-ranks, our method also produces new enumerative results (see Corollary \ref{r3} and the subsequent discussion).

Further families of quadratic fields with a given lower bound for the $3$-rank have been investigated previously by a number of authors (e.g., \cite{Craig2}, \cite{Diaz2}, \cite{KK}).  Examples of imaginary quadratic fields with $3$-rank $6$ were given by Quer \cite{Quer}, and tables of quadratic fields with large $3$-rank have been previously computed  \cite{Belabas, Diaz2, DSW, LQ}.  In particular, the smallest imaginary quadratic field with an ideal class group of $3$-rank $5$ (discovered by Quer) is known from work of Belabas \cite{Belabas}.

Our approach to Theorem \ref{mtheorem} is based on the well-known idea of using $j$-invariant $0$ elliptic curves over $\Q$ of large Mordell-Weil rank to construct quadratic fields with large $3$-rank.  There is an extensive literature detailing (and exploiting) the connection between the $3$-rank of the Selmer group attached to $3$-isogenies between such curves and the $3$-ranks of associated quadratic fields (e.g., \cite{Band, Chen, CP, Delong, Quer, Satge, Top}).  We use a geometric manifestation of this concept, combined with a method of the first author and Gillibert \cite{GL12}, and take advantage of the following chain of constructions, previously exploited in \cite{GL19}:
\begin{multline}
\text{Elliptic curve $E$ of large rank over $\Q(t)$}\stackrel{\text{$p$-descent}}{\xrightarrow{\hspace*{1.5cm}}} \text{$p$-torsion in associated Picard group}\\
 \stackrel{\text{Gillibert-Levin}}{\xrightarrow{\hspace*{1.5cm}}} \text{$p$-torsion in ideal class groups}.\label{diag}
\end{multline}

Using the multiplication-by-$p$ map on $E$, this diagram was used in \cite{GL19} to produce number fields of degree $p^2-1$ having an ideal class group with a large $p$-torsion subgroup.  In particular, taking $p=2$, in \cite{GL19} it was proven that there are infinitely many cubic number fields whose ideal class group contains a subgroup isomorphic to $(\mathbb{Z}/2\mathbb{Z})^{11}$ (see also earlier work of Kulkarni \cite{Kulkarni}).

We apply these ideas to $j$-invariant $0$ elliptic curves over $\Q(t)$ of the form
\begin{align*}
E:\ &y^2=x^3+f(t),\\
E':\ &y^2=x^3-27f(t),
\end{align*}
and use the $3$-isogeny over $\mathbb{Q}(t)$ given by
\begin{align}
\lambda:E'&\to E\label{3isog}\\
(x,y)&\mapsto \left(\frac{x^3-108f(t)}{9x^2}, \frac{(x^3+216f(t))y}{27x^3}\right)\notag.
\end{align}

When $f$ is a nonconstant square-free polynomial of degree $d$, $d\not\equiv 2,4\pmod{6}$, a descent argument yields a homomorphism
\begin{align}
\label{descentmap}
E(\Q(t))/\lambda E'(\Q(t))\hookrightarrow \Jac(C)[3](\Q),
\end{align}
where $\Jac(C)[3](\Q)$ is the rational $3$-torsion subgroup in the Jacobian of the hyperelliptic curve over $\Q$ given by the affine equation
\begin{align*}
C: y^2=f(t).
\end{align*}
Explicitly, the homomorphism \eqref{descentmap} is induced by the map
\begin{align}
E(\Q(t))\to \Jac(C)[3](\Q),\notag\\
(x_0(t),y_0(t))\mapsto \frac{1}{3}\divisor(y-y_0(t))\label{explicitmap},
\end{align}
where $\divisor(y-y_0(t))$ denotes the principal divisor on $C$ associated to the rational function $y-y_0(t)$.  In general, when $f$ is not square-free or $d\equiv 2,4\pmod{6}$, one obtains similar homomorphisms to a suitably generalized Picard group (see \cite{GHL} and \cite{GL19}).  Thus, we have described the first part of \eqref{diag} in the present context.

Next we use the idea, going back to \cite{GL12}, that rational $p$-torsion in the Jacobian of a hyperelliptic curve gives rise to $p$-torsion in ideal class groups of quadratic number fields.  The following quantitative version of this idea was proven in \cite{GL12}, explaining the application of the second map in \eqref{diag}:

\begin{theorem}
\label{hyp}
Let $C$ be a smooth projective hyperelliptic curve over $\Q$ with a rational Weierstrass point, and let $m>1$ be an integer.  Let $g$ denote the genus of $C$.  Then there exist $\gg X^{\frac{1}{2g+1}}/\log X$ imaginary (resp. real) quadratic number fields $k$ with ${|\Disc(k)|<X}$ and
\begin{align*}
& \rk_m \Cl(k)\geq \rk_m \Jac(C)(\Q)_{\tors} \\
\text{(resp. } &\rk_m \Cl(k)\geq \rk_m \Jac(C)(\Q)_{\tors}-1\text{).}
\end{align*}
\end{theorem}

Combining the above ideas gives a method for using $j$-invariant $0$ elliptic curves of large $\mathbb{Q}(t)$-rank to produce infinitely many quadratic fields with an ideal class group of large $3$-rank.

We now discuss our practical implementation of the above ideas to obtain infinitely many imaginary quadratic fields with a class group of $3$-rank at least $5$.  To begin, we use a method of Mestre \cite{Mestre} (see also \cite[\S 11]{ST}) to construct a nontrivial $j$-invariant $0$ elliptic curve $E$ over $\mathbb{Q}(t)$ with $6$ explicitly given points $P_1,\ldots, P_6\in E(\mathbb{Q}(t))$ which are independent in the Mordell-Weil group.  The elliptic curve $E$ satisfies the hypotheses giving rise to \eqref{descentmap} ($f$ is squarefree and $6|\deg f$ in our case), and thus we can use the explicit map \eqref{explicitmap} to compute the image of $P_1,\ldots, P_6$ in $\Jac(C)[3](\mathbb{Q})$.  Using Magma\footnote{Magma programs verifying the claimed calculations are available at http://users.math.msu.edu/users/adlevin/Magma.html.}\cite{Magma}, we verify that $P_1,\ldots, P_6$ yield divisor classes $[D_1],\ldots, [D_6]$ in $\Jac(C)(\mathbb{Q})$ which generate a subgroup isomorphic to $(\mathbb{Z}/3\mathbb{Z})^5$.\footnote{This is perhaps a bit surprising; na\"ively, one might expect roughly half the points to remain independent in $E(\Q(t))/\lambda E'(\Q(t))$ over $\mathbb{Z}/3\mathbb{Z}$, and we have no theoretical explanation for this advantageous situation.} Since this computation in $\Jac(C)$ is direct and independent of the above theoretical considerations (except for knowledge of the map \eqref{explicitmap}), we do not discuss or develop the material behind the homomorphism \eqref{descentmap} further.

Thus, we obtain a hyperelliptic curve $C$ over $\Q$ with $\rk_3\Jac(C)(\Q)_{\tors}\geq 5$.  Unfortunately, the curve $C$ obtained does not posses a rational Weierstrass point, and we cannot apply Theorem \ref{hyp}.  Instead, we use a result inspired by the method developed in \cite{GL19} (precisely to avoid this difficulty), which relied on finding elliptic curves $E$ over $\Q(t)$ that do not have ``universal bad reduction" (see \cite{GL19}) at any rational prime $p$.  In this direction, our primary tool for proving Theorem \ref{mtheorem} is the following result.

\begin{theorem}
\label{mtheorem2}
Let $f\in \Q[t]$ be a squarefree nonconstant polynomial of degree $d$ with $6|d$.  Let $E$ be the elliptic curve over $\Q(t)$ defined by
\begin{align*}
E: Y^2=X^3+f(t).
\end{align*}
Let $C$ be the (smooth projective) hyperelliptic curve over $\Q$ defined by $y^2=f(t)$.  Let
\begin{align*}
r=\rk_3 \im (E(\Q(t))\to \Jac(C)[3](\Q)
\end{align*}
be the $3$-rank of the image of the map \eqref{descentmap}.  In addition, suppose that
\begin{enumerate}
\item For every odd prime $p$, there exists $t_p\in \Q$ such that
\begin{align*}
3|\ord_p f(t_p).
\end{align*}
\label{h1}
\item There exists $t_2\in \Q$ such that $\Q(\sqrt{f(t_2)})$ is a quadratic extension of $\Q$ and the prime $2$ does not split in this extension.\label{h2}
\end{enumerate}
Then there exist $\gg X^\frac{1}{d}/\log X$ imaginary (resp.\ real) quadratic number fields $k$ with ${|\Disc(k)|<X}$ and
\begin{align*}
& \rk_3 \Cl(k)\geq r-\delta, \\
\text{(resp. } &\rk_3 \Cl(k)\geq r-1),
\end{align*}
where $\delta=0$ if $f$ takes negative values on $\mathbb{R}$ and $\delta=1$ otherwise.
\end{theorem}

An outline of the paper is as follows.  After recalling some basic notation and definitions, in Section \ref{sGL} we give (a slight variation on) the method of the first author and Gillibert \cite{GL12} for constructing and counting number fields with an ideal class group of large $p$-rank.  Next, we give two results describing how specializations of $\Q(t)$-points on $j$-invariant $0$ elliptic curves produce ideals which are almost perfect cubes.  In Section \ref{smain2} we prove Theorem \ref{mtheorem2}, which will be combined with constructions of Mestre (Section \ref{SMestre}, and Section \ref{shyp} in the guise of $3$-torsion in hyperelliptic Jacobians) to finish the proof of the main theorem (Theorem \ref{mtheorem}) in the final section.

Where possible, we have preferred to give reasonably elementary and self-contained proofs, both to increase the accessibility of the results, and to maintain the spirit of the origins of the work as an undergraduate research project.

%

\section{Hilbert's Irreducibility Theorem, torsion subgroups of Jacobians, and ideal class groups}
\label{sGL}

We first fix some notation.  Throughout, we let $k$ be a number field and we let $\Disc(k)$ be the (absolute) discriminant of $k$.  Let $S$ be a finite set of places of $k$ containing the archimedean places.  We let $S_{\rm fin}$ denote the subset of finite (nonarchimedean) places in $S$.  We let $\O_{k,S}$ denote the ring of $S$-integers of $k$ and let $\O_{k,S}^*$ be the group of $S$-units.  If $L$ is a finite extension of $k$, we let $S_L$ be the set of places of $L$ lying above places of $S$, and we use $\O_{L,S}$ or $\O_{L,S_L}$ to denote the ring of $S_L$-integers in $L$.  For a place $v$ of $k$ we let $|\cdot|_v$ denote a corresponding absolute value (the specific choice will not be important).  For a prime $\p$ of $\O_k$ we let $\ord_\p$ denote the associated discrete valuation.  We let $\Cl(k)$ and $\Cl(\O_{k,S})$ denote the ideal class group of $\O_k$ and $\O_{k,S}$, respectively.  For an abelian group $A$ we let $\rk A$ denote the free rank of $A$, and if $p$ is a prime and $A$ is finite, we let $\rk_p A$ be the dimension of $A/pA$ over $\mathbb{Z}/p\mathbb{Z}$. We let $H(\alpha)$ denote the absolute multiplicative height of an algebraic number $\alpha$.  If $\alpha=p/q\in\mathbb{Q}$ is written in reduced form, then $H(p/q)=\max\{\log|p|,\log|q|\}$.

We will use the following version of Hilbert's Irreducibility Theorem, combined with an enumerative result of Dvornicich and Zannier \cite{DZ}, in a form proved by Bilu and Gillibert \cite[Th.~3.1]{BG}.

\begin{theorem}[Bilu-Gillibert]
\label{HIT}
Let $k$ be a number field of degree $\ell$ over $\mathbb{Q}$.  Let $C$ be a curve over $k$ and $\phi:C\to\mathbb{P}^1$ a morphism (over $k$) of degree $d$.  Let~$S$ be a finite set of places of~$k$, $\epsilon>0$, and $\mho$ a thin subset of~$k$ \cite[\S 3.1]{BG}.  Consider  the points ${P\in C(\bar k)}$  satisfying
\begin{align*}
\phi(P)&\in k\smallsetminus \mho,\\
|\phi(P)|_v&<\epsilon, \qquad \forall v\in S,\\
H(\phi(P))&\le B.
\end{align*}
Then among the number fields $k(P)$, where~$P$ satisfies the conditions above,
there are $\gg B^\ell/\log B$ distinct fields of degree~$d$ over~$k$.
\end{theorem}


Our main tool for enumerating and constructing ideal class groups is the following result based on \cite{GL12}. We give a self-contained proof for the convenience of the reader.

\begin{theorem}
\label{thmain}
Let $C$ be a nonsingular projective curve over $\Q$.  Let $\phi\in \Q(C)$ be a nonconstant rational function on $C$ with $\deg \phi = d>1$ and let $p$ be a prime number.  Let $\psi_1,\ldots, \psi_r\in \Q(C)$ be nonconstant rational functions on $C$ whose images in $\Qbar(C)^*/(\Qbar(C)^*)^p$ generate a subgroup isomorphic to $\left(\mathbb{Z}/p\mathbb{Z}\right)^r$.  Let $S_0$ and $S$ be finite sets of places of $\Q$ and assume that $S$ contains the archimedean place.  Suppose that there exists $\epsilon>0$ and rational numbers $a_v\in \Q$, $v\in S_0$, such that for all points $P\in C(\Qbar)$ satisfying
\begin{align}
\phi(P)&\in \Q,\label{cond1}\\
|\phi(P)-a_v|_v&<\epsilon, \quad \forall v\in S_0,\label{cond2}
\end{align}
we have the equality of fractional $\O_{\Q(P),S}$-ideals
\begin{align*}
\psi_j(P)\O_{\Q(P),S}=\mathfrak{a}_{P,j}^p,
\end{align*}
for some fractional $\O_{\Q(P),S}$-ideal $\mathfrak{a}_{P,j}$, j=1,\ldots, r.  Consider the set $T_B$ of points $P\in C(\Qbar)$ satisfying \eqref{cond1}, \eqref{cond2}, and
\begin{align*}
H(\phi(P))&\leq B.
\end{align*}
Then there are $\gg B/\log B$ distinct number fields $\Q(P)$, $P\in T_B$, satisfying
\begin{align*}
[\Q(P):\Q]&=d,\\
\rk_p \Cl(\Q(P))&\geq r+ \#S_{\rm fin}-\rk \O_{\Q(P),S}^*.
\end{align*}
\end{theorem}

\begin{proof}
Let $T\subset C(\Qbar)$ consist of the set of points $P\in C(\Qbar)$ satisfying \eqref{cond1} and \eqref{cond2}.  By Kummer theory, the assumptions on $\psi_1,\ldots, \psi_r$ imply the equality
\begin{align*}
[\Qbar(C)\left(\sqrt[p]{\psi_1},\ldots,\sqrt[p]{\psi_r}\right):\Qbar(C)]=p^r.
\end{align*}
Then to the field extension $\mathbb{Q}(C)\left(\sqrt[p]{\psi_1},\ldots,\sqrt[p]{\psi_r}\right)$ of $\mathbb{Q}(C)$, we can associate a (unique up to isomorphism) nonsingular projective curve $\tilde{C}$ over $\mathbb{Q}$ and a morphism $\pi:\tilde{C}\to C$ with $\deg \pi=p^r$.  Let $\tilde{T}=\pi^{-1}(T)\subset \tilde{C}(\Qbar)$ and $\tilde{T}_B=\pi^{-1}(T_B)$. Let
\begin{align*}
L=\mathbb{Q}(\zeta_{pN})(\sqrt[p]{q}\mid q\in S_{\rm fin})=\mathbb{Q}(\zeta_{pN})(\sqrt[p]{u}\mid u\in \O_{\Q,S}^*),
\end{align*}
where $\zeta_{pN}$ is a primitive $pN$th root of unity and $N=\lcm\{n\in\mathbb{N}\mid \phi(n)\leq d\}$ (so that $L$ contains a $p$th root of every root of unity in every number field of degree $\leq d$).  Let $Q\in \tilde{T}$, $P=\pi(Q)$, and $k=\mathbb{Q}(P)$.  Assume additionally that $P$ isn't a pole of any $\psi_i$, $i=1,\ldots, r$.
\begin{lemma}
\label{deglem}
If $L$ and $\Q(Q)$ are linearly disjoint, then
\begin{align*}
[\Q(Q):\Q]\leq dp^{\rk_p\Cl(k)+\rk\O_{k,S}^*-\#S_{\rm fin}}.
\end{align*}
\end{lemma}
\begin{proof}
Since $L$ and $\Q(Q)$ are linearly disjoint, $[\Q(Q):\Q]=[L(Q):L]$, and it suffices to compute a bound for the latter degree.  Since $\deg\phi=d$, we have $[L(P):L]\leq d$ and it suffices to show that
\begin{align*}
[L(Q):L(P)]\leq p^{\rk_p\Cl(k)+\rk\O_{k,S}^*-\#S_{\rm fin}}.
\end{align*}
From the construction of $\pi$,
\begin{align}
\label{LQ}
L(Q)=L(P)\left(\sqrt[p]{\psi_1(P)},\ldots,\sqrt[p]{\psi_r(P)}\right).
\end{align}
Let $t=\rk_p \Cl(\O_{k,S})\leq \rk_p \Cl(k)$ and let $\mathfrak{b}_1,\ldots, \mathfrak{b}_t$ be $\O_{k,S}$-ideals whose ideal classes generate the $p$-torsion subgroup of $\Cl(\O_{k,S})$.  Then $\mathfrak{b}_i^p=\beta_i\O_{k,S}$ for some $\beta_i\in k$, $i=1,\ldots, t$.  Since $[k:\Q]\leq d$ and $k$ and $L$ are linearly disjoint, the only roots of unity in $k$ are $\pm 1$.  Let $t'=\rk \O_{k,S}^*$ and let $u_1,\ldots, u_{t'}$, and $-1$ be generators for $\O_{k,S}^*$.  Let
\begin{align*}
M&=k\left(\sqrt[p]{\beta_1},\ldots,\sqrt[p]{\beta_t},\sqrt[p]{u_1},\ldots,\sqrt[p]{u_{t'}},\zeta_{2p}\right),\\
M'&=k\left(\sqrt[p]{\beta_1},\ldots,\sqrt[p]{\beta_t},\sqrt[p]{u_1},\ldots,\sqrt[p]{u_{t'}}\right),
\end{align*}
for some choice of the $p$th roots (which we now fix).  For every $q\in S_{\rm fin}$, $q\in \O_{k,S}^*$, and so for some choice of $n_q\in \{0,1\}$,  $(-1)^{n_q}q$ is in the multiplicative group generated by $u_1,\ldots, u_{t'}$.  It follows that for some choice of the $p$th roots, $k(\sqrt[p]{(-1)^{n_q}q}\mid q\in S_{\rm fin})\subset   L(P)\cap M'$ and since $k$ and $L$ are linearly disjoint,
\begin{align*}
[L(P)\cap M':k]\geq [k(\sqrt[p]{(-1)^{n_q}q}\mid q\in S_{\rm fin}):k]=[\Q(\sqrt[p]{(-1)^{n_q}q}\mid q\in S_{\rm fin}):\Q]=p^{\#S_{\rm fin}}.
\end{align*}
Therefore,
\begin{align*}
[L(P)M:L(P)]&=[L(P)M':L(P)]\leq[M':L(P)\cap M']=\frac{[M':k]}{[L(P)\cap M':k]}\\
&\leq p^{t+t'-\#S_{\rm fin}}\leq p^{\rk_p\Cl(k)+\rk\O_{k,S}^*-\#S_{\rm fin}}.
\end{align*}

Then from \eqref{LQ} it suffices to show that
\begin{align*}
\sqrt[p]{\psi_i(P)}\in M, \quad i=1,\ldots, r.
\end{align*}
By hypothesis,
\begin{align*}
\psi_i(P)\O_{k,S}=\mathfrak{a}_i^p
\end{align*}
for some fractional $\O_{k,S}$-ideal $\mathfrak{a}_{i}$, $i=1,\ldots, r$.  Since $\mathfrak{a}_{i}^p$ is principal, by definition of the $\mathfrak{b}_j$ we can write
\begin{align*}
\mathfrak{a}_i=(\alpha)\prod_{j=1}^t\mathfrak{b}_j^{b_j}
\end{align*}
for some integers $b_j$ and some element $\alpha\in k$.  Therefore, raising both sides to the $p$th power, we find that
\begin{align*}
\psi_i(P)\O_{k,S}=\left(\alpha^p\prod_{j=1}^t\beta_j^{b_j}\right)\O_{k,S},
\end{align*}
and
\begin{align*}
\psi_i(P)=u\alpha^p\prod_{j=1}^t\beta_j^{b_j}
\end{align*}
for some unit $u\in \O_{k,S}^*$.  Since $u=(-1)^{c_0}\prod_{j=1}^{t'}u_j^{c_j}$ for some integers $c_j$, we find that $\sqrt[p]{\psi_i(P)}\in M$, $i=1,\ldots, r$, as desired.
\end{proof}

Let $\psi=\phi\circ \pi:\tilde{C}\to\mathbb{P}^1$ and let $\mho=\mho_1\cup \mho_2$, where
\begin{align*}
\mho_1&=\{\psi(Q)\mid Q\in \tilde{C}(\Qbar), \psi(Q)\in \mathbb{Q}, [\mathbb{Q}(Q):\mathbb{Q}]<\deg \psi\},\\
\mho_2&=\{\psi(Q)\mid Q\in \tilde{C}(\Qbar), \psi(Q)\in \mathbb{Q}, \mathbb{Q}(Q) \text{ is not linearly disjoint from $L$}\}.
\end{align*}
As in \cite[p.~947]{BG}, $\mho$ is a thin subset of $\mathbb{Q}$.  Let $a\in \mathbb{Q}$ be such that
\begin{align}
|a-a_v|_v&<\frac{\epsilon}{2}, \quad \forall v\in S_0.
\end{align}
Let $\phi'=\phi-a$.  Then $\phi(P)\in \mathbb{Q}$ if and only if $\phi'(P)\in \mathbb{Q}$, and in this case, for $v\in S_0$,  $|\phi'(P)|_v<\frac{\epsilon}{2}$ implies that $|\phi(P)-a_v|_v<\epsilon$.  We now apply Theorem \ref{HIT} to $\phi'$, $k=\mathbb{Q}$, $\mho$, and $\frac{\epsilon}{2}$.  Since $\phi$ and $\phi'$ differ by an automorphism of $\mathbb{P}^1$, $H(\phi(P))$ and $H(\phi'(P))$ differ by a bounded function.  It follows from the above that there exist $\gg B/\log B$ distinct number fields $\mathbb{Q}(P)$, where $P\in T_B$, such that if $Q\in \tilde{T}_B$, $\pi(Q)=P$, then
\begin{align*}
[L(Q):L]=[\Q(Q):\Q]=\deg \psi=(\deg \pi)(\deg \phi)=dp^r.
\end{align*}
On the other hand, by Lemma \ref{deglem}, for such points
\begin{align*}
[\Q(Q):\Q]\leq dp^{\rk_p\Cl(\mathbb{Q}(P))+\rk\O_{\mathbb{Q}(P),S}^*-\#S_{\rm fin}}.
\end{align*}
Combining these two statements yields the desired result.

\end{proof}

\section{Points on $j$-invariant $0$ elliptic curves over $\mathbb{Q}(t)$ and cubes of ideals}
\label{theory}

If $E$ is an elliptic curve over $k(t)$ given by  $Y^2=X^3+f(t)$ and $(X_0(t),Y_0(t))\in E(k(t))$ with $f$ squarefree and either $\deg f$ odd or $6|\deg f$, we first show that for all $t\in k$, $\sqrt{f(t)}-Y_0(t)$ is almost a perfect cube in $k(\sqrt{f(t)})$.  This is immediate from Weil's ``theorem of decomposition" \cite[2.7.15]{BG}, as from preceding discussions the principal divisor associated to $\sqrt{f(t)}-Y_0(t)$ (viewed as a function on the appropriate hyperelliptic curve) is divisible by $3$ in the group of divisors.  We give an elementary proof, which also has the advantage of explicitly bounding the deviation from being a cube (i.e., the set $S$ in Theorem \ref{Scube} below is explicitly constructed in the proof).

\begin{theorem}
\label{Scube}
Let $k$ be a number field.  Let $f\in k[t]$ be a nonconstant squarefree polynomial such that either $\deg f$ is odd or $6|\deg f$.  Let $E$ be the elliptic curve over $k(t)$ defined by
\begin{align*}
E: Y^2=X^3+f(t).
\end{align*}
Let $(X_0(t),Y_0(t))\in E(k(t))$.  For each $t\in k$, let $y_t=\sqrt{f(t)}$, for some choice of the square root.  Then there exists a finite set of places $S$ of $k$, containing the archimedean places, such that for all $t\in k$, we have an equality of fractional $\O_{k(y_t),S}$-ideals
\begin{align*}
(y_t-Y_0(t))\O_{k(y_t),S}=\mathfrak{a}_t^3,
\end{align*}
for some fractional $\O_{k(y_t),S}$-ideal $\mathfrak{a}_t$.
\end{theorem}

\begin{proof}
We may write
\begin{align*}
X_0(t)=\frac{a(t)}{d(t)^2}, \quad Y_0(t)=\frac{b(t)}{d(t)^3},
\end{align*}
for some polynomials $a(t),b(t),d(t)\in k[t]$ such that $a(t)b(t)$ and $d(t)$ are coprime polynomials.  Since $f(t)$ is squarefree, it is also clear that $a(t)$, $b(t)$, and $f(t)$ are pairwise coprime.  Thus, there exist polynomials $g(t), h(t)\in k[t]$ such that
\begin{align}
\label{coprime}
a(t)g(t)+f(t)h(t)=1.
\end{align}
Let $S$ be a finite set of places of $k$ containing the archimedean places such that $a,b,d,f,g,h\in \O_{k,S}[t]$, $2\in \O_{k,S}^*$, and the leading coefficients of $a,b,d$, and $f$ are $S$-units.

Let $t\in k$ and let $\p$ be a prime of $\O_{k(y_t)}$ not lying above a prime of $S$.  Then we need to show that $3|\ord_{\p}(y_t-Y_0(t))$.  Since $3|\ord_{\p}(X_0(t)^3)$ and
\begin{align*}
-X_0(t)^3=f(t)-Y_0(t)^2=(y_t-Y_0(t))(y_t+Y_0(t)),
\end{align*}
we find that $$3|\ord_{\p}((y_t-Y_0(t))(y_t+Y_0(t))=\ord_{\p}(y_t-Y_0(t))+\ord_{\p}(y_t+Y_0(t)).$$

From this equality, it clearly suffices to prove that  $$3|\min\{\ord_{\p}(y_t-Y_0(t)),\ord_{\p}(y_t+Y_0(t))\}.$$  By hypothesis, $2\in \O_{k,S}^*$ and so $\ord_{\p}(2)=0$. Then it is elementary that
\begin{align*}
\min\{\ord_{\p}(y_t-Y_0(t)),\ord_{\p}(y_t+Y_0(t))\}= \min\{\ord_{\p}y_t,\ord_{\p}Y_0(t)\}.
\end{align*}

Moreover,
\begin{align*}
\min\{\ord_{\p}y_t,\ord_{\p} Y_0(t)\}&=\frac{1}{2}\min\{\ord_{\p}y_t^2,\ord_{\p} Y_0(t)^2\}=\frac{1}{2}\min\{\ord_{\p}f(t),\ord_{\p} (X_0(t)^3+f(t))\}\\
&=\frac{1}{2}\min\{\ord_{\p}f(t),\ord_{\p} X_0(t)^3\}\\
&=\frac{1}{2}\min\{\ord_{\p}f(t),3\ord_{\p} X_0(t)\}.
\end{align*}

Hence it suffices to prove that
\begin{align}
\label{3div}
3|\min\{ \ord_{\p} f(t),3\ord_{\p}X_0(t)\}.
\end{align}

This clearly holds when $\ord_{\p}f(t)\geq 3\ord_{\p}X_0(t)$, and so we now assume that
\begin{align}
\label{fXineq}
\ord_{\p}f(t)<3\ord_{\p}X_0(t).
\end{align}
We consider three cases depending on $\ord_{\p}t$ and $\deg f$.

\textbf{Case 1}: $\ord_{\p}t\geq 0$.

Since $a,d,f,g,h\in \O_{k,S}[t]$, the quantities $\ord_{\p}a(t), \ord_{\p}d(t), \ord_{\p}f(t), \ord_{\p}g(t), \ord_{\p}h(t)$ are all nonnegative. It follows from \eqref{fXineq} that
\begin{align*}
\ord_{\p}X_0(t)>0,
\end{align*}
and as $X_0(t)=\frac{a(t)}{d(t)^2}$, we also have
$$\ord_{\p} a(t)=\ord_{\p}X_0(t)+2\ord_{\p}d(t)\ge \ord_{\p}X_0(t)>0.$$
Since $$\ord_{\p}(a(t)g(t)+f(t)h(t))=\ord_{\p}1=0,$$
we must have $\ord_{\p}f(t)=0$, proving \eqref{3div}.

\textbf{Case 2}: $\ord_{\p}t<0$ and $3|\deg f.$

Since $f\in \O_{k,S}[t]$ and the leading coefficient of $f$ is an $S$-unit, $\ord_{\p}t<0$ implies that
\begin{align*}
\ord_{\p}f(t)=(\deg f)\ord_{\p}t.
\end{align*}
By hypothesis $3|\deg f$ and so $3|\ord_{\p}f(t)$ as desired.


\textbf{Case 3}: $\ord_{\p}t<0$ and $2,3\nmid \deg f$.

The divisibility assumptions on $\deg f$ imply that $\deg fd^6\ne \deg b^2$ and $\deg fd^6 \ne \deg a^3$. Since $b^2=a^3+fd^6$, it follows that $$\deg a^3 =\deg b^2>\deg fd^6.$$

%

Hence $$3\deg a=2\deg b>\deg f +6\deg d.$$
Since $a,d,f\in \O_{k,S}[t]$, the leading coefficients of $a,d,f$ are $S$-units, and $\ord_{\p}t<0$, we have the identities
\begin{align*}
\ord_{\p}a(t)&=(\deg a)\ord_{\p}t,\\
\ord_{\p}d(t)&=(\deg d)\ord_{\p}t,\\
\ord_{\p}f(t)&=(\deg f) \ord_{\p}t.
\end{align*}
Hence, $$3\ord_{\p}X_0(t)=(3\deg a-6\deg d)\ord_{\p}t<(\deg f)\ord_{\p}t=\ord_{\p}f(t),$$
contradicting \eqref{fXineq}.

Thus, in all cases we have proven \eqref{3div}.
\end{proof}

The next result gives a way to handle the odd primes in $S$ in Theorem \ref{Scube}.  The hypothesis \eqref{avoidhyp} below is closely related to the condition from \cite{GL19} that $E$ doesn't have ``universal bad reduction" at $\p$  (if $\p$ doesn't lie above $2$ or $3$, then in our case $E$ doesn't have universal bad reduction at $\p$ if and only if there exists $t_{\p}\in k$ such that $6|\ord_{\p} f(t_{\p})$).

\begin{theorem}
\label{avoid}
Let $k$ be a number field.  Let $f\in k[t]$ be a nonconstant polynomial.  Let $E$ be the elliptic curve over $k(t)$ defined by
\begin{align*}
E: Y^2=X^3+f(t).
\end{align*}
Let $(X_0(t),Y_0(t))\in E(k(t))$.  For each $t\in k$, let $y_t=\sqrt{f(t)}$, for some choice of the square root.  Let $\p$ be a prime of $\O_k$ not lying above $2$.  Suppose that there exists $t_{\p}\in k$ such that
\begin{align}
\label{avoidhyp}
3|\ord_{\p} f(t_{\p}).
\end{align}
Then there exists $\epsilon_{\p}>0$ such that if $t\in k$ and $|t-t_{\p}|_{\p}<\epsilon_{\p}$, then for any prime $\mathfrak{q}$ of $k(y_t)$ lying above $\p$, we have
\begin{align*}
3|\ord_{\mathfrak{q}}(y_t-Y_0(t)).
\end{align*}
\end{theorem}

\begin{proof}
By continuity, for all $t\in k$ sufficiently $\p$-adically close to $t_{\p}$, we have $\ord_{\p}f(t)=\ord_{\p}f(t_{\p})$ and $3|\ord_{\p}f(t)$ (and $3|\ord_{\mathfrak{q}}f(t)$ for any prime $\mathfrak{q}$ of $k(y_t)$ lying above $\p$).  Since $\p$ does not lie above $2$, the result follows from the equivalence with the condition \eqref{3div} used in the proof of Theorem \ref{Scube}.
\end{proof}

\section{Proof of Theorem \ref{mtheorem2}}
\label{smain2}

We now have all of the tools to prove Theorem \ref{mtheorem2}.

\begin{proof}[Proof of Theorem \ref{mtheorem2}]
Let $P_i=(X_i(t),Y_i(t))\in E(\Q(t))$, $i=1,\ldots, r$, be points whose images are independent (over $\mathbb{Z}/3\mathbb{Z}$) in $\Jac(C)[3](\Q)$.   Let $S_0$ consist of the prime $2$ and the (finite) union of the set of places of $\Q$ given by Theorem \ref{Scube} (with $k=\Q$) applied to $P_i$, $i=1,\ldots, r$.  By hypothesis, for each odd prime $p\in S_0$, there exists $t_p\in \Q$ such that $3|\ord_p f(t_p)$.  For each such $p$ and $t_p$, let $\epsilon_p>0$ be as in Theorem \ref{avoid}. By hypothesis, there exists $t_2\in \Q$ such that $\Q(\sqrt{f(t_2)})$ is a quadratic extension of $\Q$ and the prime $2$ does not split in this extension.  Then for some sufficiently small $\epsilon_2>0$, by Krasner's lemma (or elementary arguments in this specific case) if $|t-t_2|_2<\epsilon_2$ then $\Q(\sqrt{f(t)})$ is a quadratic extension of $\Q$ and the prime $2$ does not split in this extension.  Let $t_\infty\in \Q$ and $\epsilon_\infty>0$ be such that $f(t)$ is negative (resp. positive) if $|t-t_\infty|<\epsilon_\infty$.  Let $\epsilon=\min_{p\in S_0}\epsilon_p$. Let $S=\{2,\infty\}$.  Consider the rational functions $\psi_i=y-Y_i(t)\in \Q(t,y)=\Q(C)$, $i=1,\ldots, r$.  By Theorem \ref{Scube} and Theorem \ref{avoid}, if $P\in C(\Qbar)$, $t=t(P)\in \Q$ and $|t-t_p|_p<\epsilon$ for all $p\in S_0$, then $\psi_i(P)=\sqrt{f(t)}-Y_i(t)$ generates the cube of a fractional ideal outside $2$, i.e.,  $\psi_i(P)\O_{\Q(P),S}=\mathfrak{a}_{P,i}^3$ for some fractional $\O_{\Q(P),S}$-ideal $\mathfrak{a}_{P,i}$, $i=1,\ldots, r$.  Furthermore, from our assumptions at $p=2,\infty$, $\Q(P)$ is an imaginary (resp.~real) quadratic field in which $2$ does not split.  Note also that if $t\in \Q$ and $H(t)<B$, then the discriminant of $\Q(\sqrt{f(t)})$ is $\ll B^d$ (we use here that $d$ is even).  By our hypotheses and the form of the map \eqref{descentmap}, the functions $\psi_1,\ldots, \psi_r$ generate a subgroup isomorphic to $\left(\mathbb{Z}/3\mathbb{Z}\right)^r$ in $\Qbar(C)^*/(\Qbar(C)^*)^3$.  Let $\phi=t$, a rational function of degree $2$ on $C$.  Combining all of the above, Theorem \ref{thmain} gives that there are $\gg X^\frac{1}{d}/\log X$ distinct imaginary (resp.~real) quadratic number fields $k$ with ${|\Disc(k)|<X}$ satisfying
\begin{align*}
\rk_3 \Cl(\Q(P))&\geq r+ \#S_{\rm fin}-\rk \O_{\Q(P),S}^*.
\end{align*}
Since $S=\{2,\infty\}$, $\#S_{\rm fin}=1$, and since there is a single prime lying above $2$ in $\O_{\Q(P)}$, we have $\rk \O_{\Q(P),S}^*=1$ (resp.~$\rk \O_{\Q(P),S}^*=2$).  This gives the desired result when $f$ takes both negative and positive values on $\mathbb{R}$.  Otherwise, we obtain the appropriate result for either real or imaginary quadratic fields, and the full theorem follows from Scholz's reflection principle \cite{Scholz}: if $d$ is a positive integer, then $\rk_3\Cl(\sqrt{-3d})-\rk_3\Cl(\sqrt{d})\in \{0,1\}$.
\end{proof}


\section{Mestre's construction of a $j$-invariant $0$ elliptic curve over $\mathbb{Q}(t,u)$ of rank at least $6$}
\label{SMestre}

We recall a construction of Mestre \cite{Mestre, Mestre95} which yields a $j$-invariant $0$ elliptic curve over $\mathbb{Q}(t,u)$ of rank at least $6$ (and after a specialization, a $j$-invariant $0$ elliptic curve over $\mathbb{Q}(t)$ of rank at least $7$).

Let $K$ be a field of characteristic $0$ and let $x_1,\ldots, x_5\in K$.  Let
\begin{align*}
x_6&=-(x_1+\cdots+x_5),\\
p(X)&=(X-x_1)\cdots(X-x_6)\\
&=X^6+a_4X^4+a_3X^3+a_2X^2+a_1X+a_0,\\
\intertext{and}
g(X)&=X^2+a_4/3.
\end{align*}
Then
\begin{align*}
g(X)^3-p(X)=r(X)
\end{align*}
for some cubic polynomial $r\in K[X]$.  Explicitly,
\begin{align*}
r(X)=r_3X^3+r_2X^2+r_1X+r_0,
\end{align*}
where $r_3=-a_3,r_2=a_4^2/3-a_2,r_1=-a_1$, and $r_0=(a_4/3)^3-a_0$.  Then the cubic curve $C:Y^3=r(X)$ will generically be a genus $1$ curve of $j$-invariant $0$ and possess the six $K$-rational points $P_i=(x_i,g(x_i))$, $i=1,\ldots, 6$.  To obtain a further point when $K=\mathbb{Q}(t,u)$, Mestre shows that if
\begin{align*}
x_1&=\frac{1-u^3}{4u}+t,\\
x_2&=\frac{1-u^3}{4u}-t,\\
x_3&=-\frac{u^3+3}{4u}+t,\\
x_4&=\frac{u^6-6u^3-3}{4(u^4-u)}-t,\\
x_5&=\frac{u^9 - u^6 + 15u^3 + 1}{4(u^7 -u)}+t,\\
x_6&=-(x_1+\cdots+x_5)=\frac{u^6+8u^3-1}{4(u^4+u)}-t,
\end{align*}
then $r_3=a_3=-1$.  Then since $r_3$ is a perfect cube, the curve $C$ has a $7$th rational point $P_7$ at infinity.  Setting, say, $P_7$ as the origin, one easily finds (by a specialization and explicit computation) that the other $6$ points are independent in the Mordell-Weil group of the resulting elliptic curve over $\mathbb{Q}(t,u)$.  Mestre goes further and uses this construction to find an elliptic curve of rank at least $7$ over $\mathbb{Q}(t)$.  However, putting the curve in Weierstrass form, the extra point constructed has the same $y$-coordinate as a previous point, and so the new point does not produce any extra $3$-torsion (via the map \eqref{descentmap}).

\section{$3$-torsion in hyperelliptic Jacobians}
\label{shyp}

Let $p$ be an odd prime.  As a function field analogue of constructing quadratic fields with an ideal class group of large $p$-rank, it is an interesting problem in its own right to construct hyperelliptic curves $C$ over $\Q$ such that $\rk_p \Jac(C)[p](\Q)$ is large.  For a given $p$-rank, it is further interesting to minimize the genus of the curve $C$ involved, both for intrinsic reasons and for applications to the enumerative problem of counting quadratic fields with interesting ideal class groups.  Since $\mathbb{Q}$ doesn't contain a primitive $p$th root of unity, a well-known argument using the Weil pairing gives the bound (for any smooth projective curve $C$)
\begin{align*}
\rk_p\Jac(C)[p](\Q)\leq g(C),
\end{align*}
where $g(C)$ denotes the genus of $C$.

We now restrict to the case $p=3$.  Using the constructions of the previous section we prove the following result.

\begin{theorem}
\label{Jrank}
There exists a hyperelliptic curve $C:y^2=f(t)$ over $\mathbb{Q}$ of genus $g$, $\deg f=d$, and $\rk_3\Jac(C)[3](\Q)\geq r$ for the following values of $r$, $d$, and $g=g(C)$:\\
\begin{center}
\begin{tabular}{|c|c|c|}
\hline
r & d & g\\
\hline
$1$ & $3$ & $1$\\
\hline
$2$ & $5$ & $2$\\
\hline
$3$ & $9$ & $4$\\
\hline
$4$ & $10$ & $4$\\
\hline
$5$ & $30$ & $14$\\
\hline
\end{tabular}\\
\bigskip
\end{center}

Explicitly, the last three entries are realized by the hyperelliptic curves
\begin{align*}
C_r:y^2=f_r(t), \quad r=3,4,5,
\end{align*}
with generators for $\Jac(C)[3](\Q)$ given by the classes of
\begin{align*}
D_{r,i}=
\begin{cases}
\frac{1}{3}\dv(y-y_{r,i}(t)),  &\text{if $r=3,5$}\\
\frac{1}{3}\dv(y-y_{r,i}(t))-\frac{1}{3}\dv(y-y_{r,5}(t)), &\text{if $r=4$}
\end{cases}
\end{align*}
for $i=1,\ldots, r$, where
\begin{align*}
f_3(t)=&t^{9} + 2973t^{6}-  369249t^3 + 11764900\\
y_{3,1}(t)=&t^6 - 106t^3 + 3430\\
y_{3,2}(t)=&\frac{1}{64}t^6 + \frac{269}{4}t^3 + 3430\\
y_{3,3}(t)=&t^6 + 36t^5 + 486t^4 + 3350t^3 + 13914t^2 + 33264t + 40474
\end{align*}
{\tiny
\begin{align*}
f_4(t)=&127358629188153017343112694654244t^{10} - 14476726558441542259500980593582900t^9\\
    &+767540949843094964859507359162484321t^8 - 23227949011157855871750302161149318060t^7\\
    &+486933739385947419621206507920009537350t^6 - 8471956828413213486742748322179256745500t^5\\
    &+139665528153448288531118705650287136663899t^4 - 1509800364506319291441531124462079071041720t^3\\
    &+14597743197263467927181474503046907251979462t^2 - 135004259433655686521826532061360904927910680t\\
    &+543592155691663960065241800360826161610140961\\
y_{4,1}(t)=&11285328049647162t^5 - 612703879315493343t^4 + 11259180860536474740t^3 + 124175441794992816207t^2\\
&-2894185136924624900028t + 23315023973129417008893\\
y_{4,2}(t)=&11285328049647162t^5 - 637837872300913137t^4 + 15898504501345253760t^3 - 133140352448943347487t^2\\
    &+2895345088136314829232t - 23315023973129417008893\\
y_{4,3}(t)=&11285328049647162t^5 - 641395912765696179t^4 + 15779426445792454284t^3 - 132367737077167916373t^2\\
    &+2891689893601551455028t - 23295069079544963156463\\
y_{4,4}(t)=&11285328049647162t^5 - 787772412770249571t^4 + 14665379977955069244t^3 - 219254957235699134757t^2\\
    &+2621647313739427449588t - 35358708563462647994607\\
y_{4,5}(t)=&11285328049647162t^5 + 2985212511317920527t^4 + 74598289369102611840t^3 + 1831498005344531215377t^2\\
    &+17007935571912588694272t + 183122730884960782522323	
\end{align*}
}
and
\begin{align*}
f_5(t)=&f_4(t^3),\\	
y_{5,i}(t)=&y_{4,i}(t^3), \quad i=1,\ldots, 5.
\end{align*}

\end{theorem}

We note that, in terms of the genus, the results for $r=1,2,4$ are sharp. We do not know if there exists a hyperelliptic curve with $(r,g)=(3,3)$.  However, it was shown in \cite{HLP} that there exist (non-hyperelliptic) curves $C$ of genus $3$ with $\rk_3\Jac(C)[3](\Q)=3$.

Combined with Theorem \ref{hyp}, the result for $r=3$ immediately gives:

\begin{corollary}
\label{r3}
There exist $\gg X^\frac{1}{9}/\log X$ imaginary quadratic number fields $k$ with ${|\Disc(k)|<X}$ and
\begin{align*}
& \rk_3 \Cl(k)\geq 3.
\end{align*}
\end{corollary}

With further work (as in the next section) the result for $r=4$ should also lead to new enumerative results, although we did not pursue this.  For smaller $3$-rank, the best enumerative results for ideal class groups of quadratic fields are due to Heath-Brown \cite{HB2, HB1} (rank $1$, i.e., class number divisible by $3$) and Luca and Pacelli \cite{LP} (rank $2$), with many earlier results by other authors (e.g., \cite{BK}, \cite{CR}, \cite{Sound}).

\begin{proof}[Proof of Theorem \ref{Jrank}]
The first two entries in the table are well-known.  For any value of $a\in\Q\setminus\{ 0,\frac{1}{27}\}$, the elliptic curve $y^2=t^3+(t+a)^2/4$ has the $3$-torsion point $(0,a/2)$. For any value of $a,b\in\Q$, $a^2-4b^2\neq 0, b\neq 0$, the genus two curve $C_{a,b}:y^2=t^6+at^3+b^2$ satisfies $\rk_3\Jac(C_{a,b})[3](\Q)=2$, with generators for $\Jac(C_{a,b})[3](\Q)$ given by the classes of $\frac{1}{3}\dv(y-t^3\pm b)$. For appropriate choices of $a$ and $b$, $C_{a,b}$ possesses a rational Weierstrass point, and hence also admits a quintic Weierstrass model over $\Q$.

In the case $r=3$, we consider Mestre's construction in Section \ref{SMestre} with the parameters
\begin{align*}
(x_1,x_2,x_3,x_4,x_5,x_6)=(2,3,-5,0,t,-t).
\end{align*}
This yields a genus one curve $Y^3=r(X)$ over $\Q(t)$ with (at least) $6$ $\Q(t)$-rational points $(x_i,g(x_i))$, $i=1,\ldots, 6$.  Choosing, say, the first point as the origin, one finds the elliptic curve
\begin{align*}
\tilde{E}_3:y^2=x^3+f_3(t^2)
\end{align*}
and $5$ points $\tilde{P}_1,\ldots, \tilde{P}_5\in \tilde{E}_3(\Q(t))$ (the curve $\tilde{E}_3$ also appears in \cite{ST}).  The points $\tilde{P}_1,\tilde{P}_2,\tilde{P}_3,\tilde{P}_4+\tilde{P}_5\in \tilde{E}_3(\Q(t))$ have coordinates that are rational functions in $t^2$, and thus we find $4$ corresponding points, call them $P_1,\ldots, P_4$, on the elliptic curve
\begin{align*}
E_3:y^2=x^3+f_3(t).
\end{align*}
Let $y_{3,i}(t)=y(P_i)\in \Q(t)$ be the $y$-coordinates of the $4$ points.  Abusing notation, we also view $y_{3,i}(t)\in \Q(C_3)=\Q(t,y)$ and consider the rational function $y-y_{3,i}(t)$ on the hyperelliptic curve $C_3$.  By explicit computation, one verifies that $\dv(y-y_{3,i}(t))=3D_{3,i}$ for some divisor $D_{3,i}$ on $C_3$, $i=1,\ldots, 4$.  Thus, the divisors $D_{3,i}$ yield $3$-torsion classes $[D_{3,i}]\in \Jac(C_3)[3](\Q)$, $i=1,\ldots, 4$.  By explicit computation (with Magma), one finds that $[D_{3,1}], [D_{3,2}]$, and $[D_{3,3}]$ are independent (over $\mathbb{Z}/3\mathbb{Z})$, while $[D_{3,1}]+\cdots +[D_{3,4}]=0$.  Therefore $\rk_3\Jac(C_3)[3](\Q)\geq 3$, giving the $3$rd entry in the table.

Finally, we describe the (related) constructions for $r=4,5$.  The construction from Section \ref{SMestre} yields an elliptic curve over $\Q(t,u)$ with $6$ $\Q(t,u)$-points that are independent in the Mordell-Weil group.  Specializing to $u=2$ and rescaling $t$ (to simplify the resulting polynomials) we find the elliptic curve
\begin{align*}
E_4:y^2=x^3+f_4(t),
\end{align*}
with six corresponding points $P_1,\ldots, P_6\in E_4(\Q(t))$ (the sixth point turns out not to give new torsion).  We set $y_{4,i}(t)=y(P_i)\in \Q(t)$ , $i=1,\ldots, 5$, and view $y_i(t)\in \Q(C_4)=\Q(t,y)$.  Define the divisors
\begin{align*}
D_{4,i}=\frac{1}{3}\dv(y-y_{4,i}(t))-\frac{1}{3}\dv(y-y_{4,5}(t)), \quad i=1,\ldots, 4
\end{align*}
on $C_4$.  Using Magma, one verifies that the divisor classes $[D_{4,1}],\ldots, [D_{4,4}]$ are independent over $\mathbb{Z}/3\mathbb{Z}$.  Therefore $\rk_3\Jac(C_4)[3](\Q)\geq 4$.  Finally, with
\begin{align*}
f_5(t)=&f_4(t^3),\\	
y_{5,i}(t)=&y_{4,i}(t^3), &&i=1,\ldots, 5\\
D_{5,i}=&\frac{1}{3}\dv(y-y_{5,i}(t)), &&i=1,\ldots, 5,
\end{align*}
a calculation in Magma with $[D_{5,1}],\ldots, [D_{5,5}]$ shows that $\rk_3\Jac(C_5)[3](\Q)\geq 5$.
\end{proof}

\section{Proof of Theorem \ref{mtheorem}}

The proof of Theorem \ref{mtheorem} is based on Theorem \ref{mtheorem2} and a slight modification of the construction from last section when $r=5$ (so that the hypotheses of Theorem \ref{mtheorem2} are appropriately satisfied).

\begin{proof}[Proof of Theorem \ref{mtheorem}]

We use Mestre's construction from Section \ref{SMestre}, with an appropriately chosen specialization so that \eqref{h1} and \eqref{h2} of Theorem \ref{mtheorem2} hold and $f$ takes negative values on $\mathbb{R}$.  For this purpose, letting $u=-\frac{9}{5}$ and replacing $t$ by $t^3$ in the constructions of Section \ref{SMestre}, we find an elliptic curve
\begin{align*}
E:y^2=x^3+f(t),
\end{align*}
with $f\in \mathbb{Z}[t]$, $\deg f=30$, and six points $P_1,\ldots, P_6\in E(\Q(t))$. Let $y_i(t)=y(P_i)\in \Q(t)$ , $i=1,\ldots, 6$.  Let $C$ be the hyperelliptic curve $y^2=f(t)$ and view $y_i(t)\in \Q(C)=\Q(t,y)$.  Using Magma, one verifies that $\dv(y-y_i(t))=3D_i$ for some divisor $D_i$ on $C$, $i=1,\ldots, 6$, and that $[D_1],\ldots, [D_5]$ generate a subgroup $\left(\mathbb{Z}/3\mathbb{Z}\right)^5$ in $\Jac(C)(\Q)$.  Moreover, the leading coefficient of $f$ is positive, $f(0)$ is negative, $2$ and $5$ are the only prime factors of $\gcd(f(0), f(1))$, $\ord_2(f(0))=3$ and $\ord_5(f(0))=12$. Then \eqref{h1} of Theorem \ref{mtheorem2} holds (for some choice of $t_p\in\{0,1\}$ for every odd prime $p$), and \eqref{h2} holds as $2$ ramifies in $\Q(\sqrt{f(0)})$.  Since $\deg f=d=30$ and $f$ takes negative values on $\mathbb{R}$, the conclusion of Theorem \ref{mtheorem2} gives the desired statement.
\end{proof}

\bibliography{classgroup}

\end{document}